\documentclass[11pt,reqno]{amsart} % please use amsart at 11pt

\usepackage{amssymb,latexsym,color}
\usepackage{cite} % to get Refs. [1,2,3] typeset as [1--3] automatically
\usepackage{graphicx}
\theoremstyle{plain}
\newtheorem{theorem}{Theorem}[section]

\theoremstyle{definition}

\theoremstyle{remark}

\numberwithin{equation}{section} % to get equations numbered
\begin{document}
\title[On Star Coloring of Splitting Graphs]{On Star Coloring of Splitting Graphs} 

\author{Hanna Furma\'nczyk}
\address{Institute of Informatics \\ University of Gda\'nsk \\ Wita Stwosza 57 \\ 80-952 Gda\'nsk \\ Poland}
\email{hanna@inf.ug.edu.pl}

\author{Kowsalya.V}
\address{Part-Time Research Scholar (Category-B) \\ Research \& Development Centre \\ Bharathiar University \\ Coimbatore 641 046 \\ Tamilnadu\\ India}
\email{vkowsalya09@gmail.com}																

\author{Vernold Vivin.J}
\address{Department of Mathematics\\ University College of Engineering Nagercoil\\ (Anna University Constituent College)\\ Konam \\ Nagercoil - 629 004\\  Tamilnadu\\  India}
\email{vernoldvivin@yahoo.in}

\begin{abstract}
In this paper, we consider the problem of a star coloring. In general case the problems in NP-complete. 
We establish the star chromatic number for
splitting graph of complete and complete bipartite graphs, as well of paths and cycles. Our proofs are constructive, so they lead to appropriate star colorings of graphs under consideration.
\end{abstract}



\maketitle

\section{Introduction}
\noindent
\par We consider only finite, undirected, loopless graphs without multiple edges.
\par The notion of star chromatic number was introduced by Branko Gr\"{u}nbaum in 1973. A \emph{star coloring} \cite{alberton, bg, fertin} of a graph $G$ is a proper vertex coloring in which 
every path on four vertices uses at least three distinct colors. Equivalently, in a star coloring, the induced subgraphs formed by the vertices of any two color classes has connected components 
that are star graphs. The {\emph{ star graph}} is a tree with at most one vertex with degree larger than $1$.
Star coloring is a strengthening of \emph{acyclic coloring} \cite{bg}, i.e. proper coloring in which every two color classes induce a forest.
The \emph{star chromatic number} $\chi_s\left(G\right)$ of $G$ is the least number of colors needed to star coloring of $G$.

Guillaume Fertin et al.\cite{fertin} gave the exact value of the star chromatic number of different families of graphs such as trees, cycles, complete bipartite graphs, outerplanar graphs, 
and $2$-dimensional grids. They also investigated and gave bounds for the star chromatic number of other families of graphs, such as planar graphs, hypercubes, $d$-dimensional grids 
$(d\geq3)$, $d$-dimensional tori $(d\geq2)$, graphs with bounded treewidth, and cubic graphs.

In this paper we consider star coloring of some splitting graphs. For a given 
graph $G$ the \emph{splitting graph} 
\cite{sam} $S(G)$ of graph $G$ is obtained by adding a new vertex $v'$ corresponding to each vertex $v$
of $G$ such that $N(v)=N(v')$, where $N(x)$ is the neighborhood of vertex $x$. 
For example, the splitting graph of $K_{2,3}$ is given in Fig. \ref{bip}.

Albertson et al.\cite{alberton} showed that it is NP-complete to determine whether $\chi_s\left(G\right)\leq 3$, even when $G$ is a graph that is both planar and bipartite. 
Coleman and Mor\'{e} \cite{col} proved that finding an optimal star 
coloring is NP-hard and remain so even for bipartite graphs. 

\begin{center}
\begin{figure}[htb]
\includegraphics[scale=0.4]{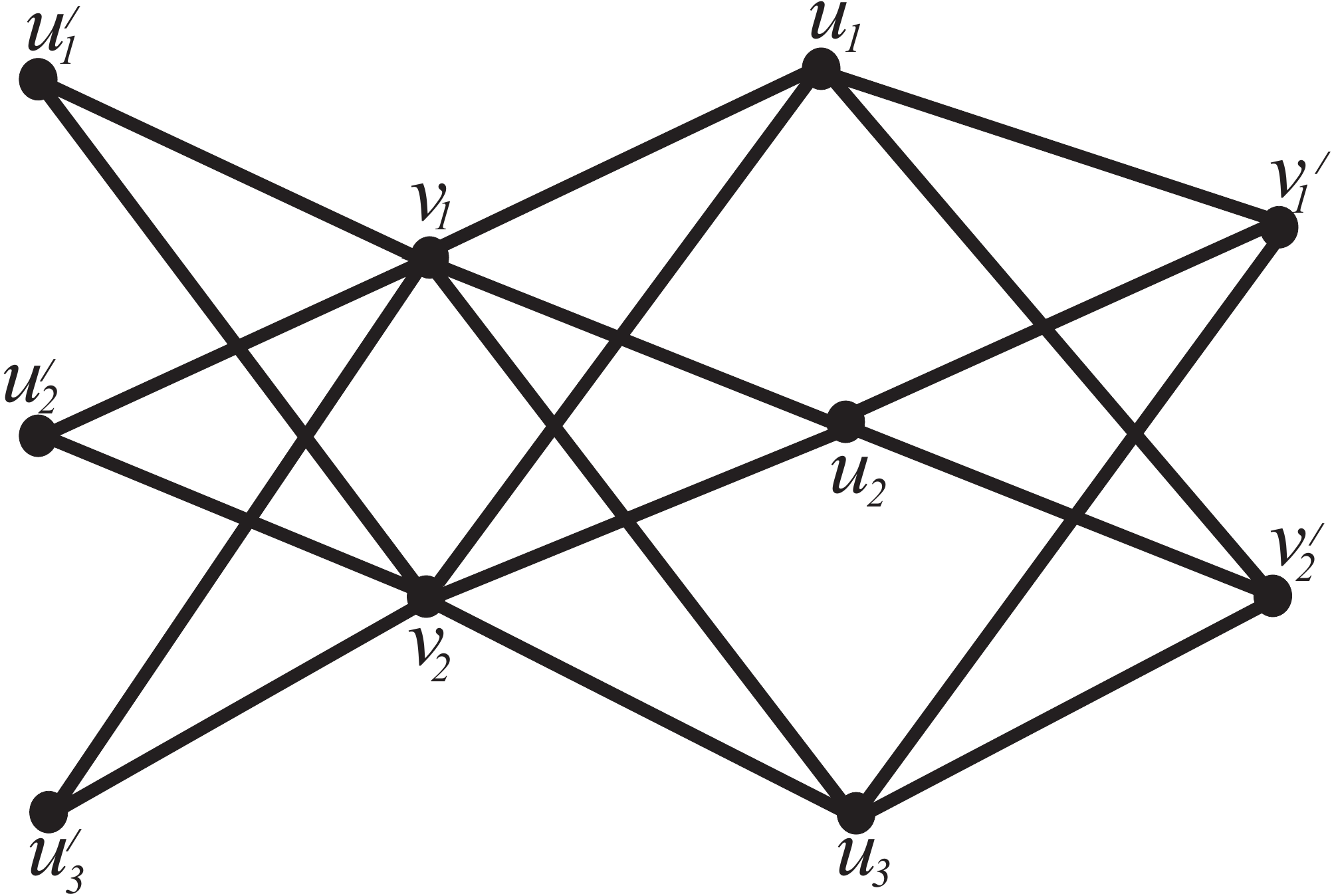}
\caption{$S(K_{2,3})$.}\label{bip}
\end{figure}
\end{center}

One can ask whether there is a subclass of planar graphs such that admits optimal star coloring in polynomial time. 
Sampathkumar and Walikar \cite{sam} posed an open problem: full characterization of graphs whose splitting graphs are planar. In this situation considering star coloring of splitting graphs 
seems to be desirable. Moreover, star coloring problem has application in combinatorial scientific computing.
In particular, it has been employed since 1980's to efficiently 
compute sparse Jacobian and Hessian matrices using either finite 
differences or automatic differentiation \cite{col}.
%\section{Preliminaries}

Additional graph theory terminology used in this paper can be found in \cite{bm, f}.

For the completness of the reasoning given in this paper, we recall some known results. 

\begin{theorem}[\cite{fertin}]\label{fertin1}
	If $C_n$ is a cycle on $n \geq 3$ vertices, then
	$$\chi_s(C_n)=\begin{cases}4\quad when\quad n=5\\
	3 \quad otherwise.
	\end{cases}$$
\end{theorem}

%\begin{theorem}[\cite{fertin}] \label{fertin2}
%Let $K_{n,m}$ be a complete bipartite graph, $n, m \geq 1$. Then $\chi_s(K_{n,m})= min\left\{m,n\right\}+1$.
%\end{theorem}

In this paper we prove results concerning the star chromatic number of splitting graph of complete graphs, 
paths, complete bipartite graphs and  cycles.

\section{Main Results} 

\subsection{Star Coloring of Splitting graph of complete graphs}

\begin{theorem}
Let $K_n$ be a complete graph on $n\geq2$ vertices. Then $$\chi_s(S(K_n))=n+1.$$
\end{theorem} 

\begin{proof}
Let $V(S(K_n))=\left\{v_i:1\leq i \leq n \right\} \cup \left\{v_i^\prime:1 \leq i \leq n \right\}.$ 
We define star $(n+1)$-coloring $\sigma$ of $S(K_n)$ in the following way:
$$\sigma(v_i)=c_i:1\leq i\leq n; \sigma(v^\prime_i)=c_{n+1}:1\leq i\leq n.$$
Clearly the vertices 
of complete graph $K_n$ needs $n$ colors for a proper coloring and hence for star coloring. Thus, 
$\chi_s(S(K_n))\geq n$. In the further part we will show that $n$ colors are insufficient for star coloring of $S(K_n)$.

By definition of splitting graph, for $1\leq j\leq n$, the vertex $v_j$ is adjacent to all $v_i$ except for 
$i=j$. Hence $\sigma (v_j)\neq \sigma(v_i^{\prime})$ for $i \neq j$. If  $\sigma(v_i^\prime)=\sigma(v_i)$, 
then there exist bicolored paths $v_i^\prime v_{i+1}v_iv_{i+1}^\prime$ for $1\leq i \leq n$. 
A contradiction to proper star coloring. Thus, $\chi_s(S(K_n))=n+1$. 
\end{proof}

\subsection{Star coloring of splitting graph of paths} 

\begin{theorem}
Let $P_n$ be a path on $n\geq 4$ vertices. Then $$\chi_s(S(P_n))=4.$$
\end{theorem}

\begin{proof}
Let $V(S(P_n))=\left\{v_1,v_2,\ldots,v_n, v_1', \ldots, v_n'\right\}$.  
	
It is clear that vertices of $P_n$ need three colors for proper star coloring. Thus $\chi_s(S(P_n))\geq 3$.

Now, we will show that three colors are insufficient for star coloring of $S(P_n)$, $n\geq 4$. We start from 
any star 3-coloring of $P_n$. we will denote it by the coloring $c$. Let $v_i, v_{i+1}, v_{i+2}, v_{i+3}$, $1 \leq i \leq n-3$, be any four consecutive 
vertices of $P_n$. It is clear that there exists at least one pair of vertices of length two among these four ones with different colors assigned. 
Without lost of generality we may assume that $c(v_i) \neq c(v_{i+2})$. Then color that may be used to color vertex $v_{i+1}^{\prime}$ is determined to 
$c(v_{i+1})$. Let us try to assign appropriate color to $v_{i+2}^{\prime}$. We have two possibilities - we can use color $c(v_{i})$ or $c(v_{i+2})$. But any 
of these two choices leads to 2-chromatic $P_4$. Thus $\chi_s(S(P_n))\geq 4$.

The star $4$-coloring $\sigma$ of $S(P_n)$ is defined as follows:
	$$\sigma(v_i)= \begin{cases} 
	c_1 \;\; \mbox{if} \;\; i\equiv 1 \bmod 3 \\ 
	c_2 \;\; \mbox{if} \;\; i\equiv 2 \bmod 3 \\ 
	c_3\;\; \mbox{if} \;\;  i\equiv 0 \bmod 3  
	\end{cases}$$
and 	
	$$\sigma(v^\prime_i)=c_4$$
for $1\leq i\leq n$.
It is easy to verify that any two color classes in coloring $\sigma$ induce a forest whose components are $K_{1,2}$ and $K_2$, 
hence by definition $\sigma$ is a proper star 4-coloring. Hence, $\chi_s(S(P_n))=4$.
\end{proof}

It is easy to verify that $\chi_s(S(P_2))=\chi_s(S(P_3))=3$.

\subsection{Star coloring of splitting graph of complete bipartite graphs}

\begin{theorem}
Let $K_{m,n}$, $m \leq n$ be complete bipartite graph. Then $$\chi_s(S(K_{m,n}))=2m+1.$$
\end{theorem}

\begin{proof}
Let $V(K_{m,n})=X \cup Y=\left\{v_1,v_2,\ldots,v_m\right\}\cup\left\{u_1,u_2,\ldots,u_n\right\}$. Then we have 
$V(S(K_{m,n}))=\left\{X,X',Y,Y'\right\}= \left\{v_1,v_2,\ldots, v_m ;v_1^\prime,v_2^\prime,\ldots, 
v_m^\prime; u_1,u_2,\ldots,\right.$ $\left. u_n; u_1^\prime,u_2^\prime,\ldots, u_n^\prime \right\}$.
	
Note that in any star $k$-coloring of complete graph $K_{m,n}$, $k < m+n$,
there is only one bipartite partition set, let say $A$, including at least two vertices from the same color class. Vertices in the second bipartite 
partition set, 
let say $B$, must be colored with different $|B|$ colors and these colors cannot be assigned to vertices in set $A$. We will named such multicolored 
bipartite partition set $B$ as \emph{rainbow}.

This implies that at least two out of four partition sets $X,X',Y,Y'$ of $S(K_{m,n})$ must be rainbow:  $X,Y$, $X,X'$, $Y,Y'$, or $X',Y'$. 
\begin{description}
\item[Case 1] $X$ and $Y$ are rainbow.
\\Note that such coloring of $X$ and $Y$ uses different $(m+n)$ colors and it may be extended to proper star $(m+n+1)$-coloring of $S(K_{m,n})$.

\item[Case 2] $X$ and $X'$ are rainbow.
\\Note that such coloring of $X$ and $X'$ may use $m$ colors, but then it cannot be extended to any star $k$-coloring of $S(K_{m,n})$ with $k < m+n+1$. 
We may extend it only to star $k$-coloring with $k \geq m+n+1$. Indeed, $Y$ must be also raibow in this case, otherwise 2-chromatic $P_4$ arises.
\\If coloring of $X$ and $X'$ uses $2m$ colors, then it may be extended to star $(2m+1)$-coloring. Notice that $2m+1 \leq m+n+1$ for $m \leq n$.

\item[Case 3] $Y$ and $Y'$ are rainbow.
\\This case may be considered analogously to Case 2. We have two possibilities: star $(n+m+1)$- or $(2n+1)$-coloring of $S(K_{m,n})$.

\item[Case 4] $X'$ and $Y'$ are rainbow.
\\In this cases at least one out of $X$ and $Y$ must be also rainbow partition set and we may look for optimal star coloring of $S(K_{m,n})$ due to 
Case 2 or Case 3.
\end{description}

Summarizing, star coloring of $S(K_{m,n})$ with the smallest number of colors is mentioned in Case 2. In details, we define star $(2m+1)$-coloring 
$\sigma$ of $S(K_{m,n})$ in the following way.

$$\sigma(v_i)=c_i, 1\leq i\leq m$$
$$\sigma(u_j)=c_{m+1}, 1\leq j\leq n$$
$$\sigma(u^\prime_j)=c_{m+1}, 1\leq j\leq n$$ 
$$\sigma(v_i^\prime)=c_{m+1+i}, 1\leq i\leq m.$$
\end{proof}

\subsection{Star coloring of splitting graph of cycles}

\begin{theorem}
Let $C_n$ be a cycle graph on $n\geq 3$ vertices. Then $$\chi_s(S(C_n))
\begin{cases}
= 4 \;\; \mbox{if} \;\; n \not \equiv 1 \bmod 3 \text{ and } n \neq 5\\
\leq 5 \;\; \mbox{otherwise}
\end{cases}$$
\end{theorem}

\begin{proof}
Let $V(C_n)=\left\{v_1,v_2,\ldots, v_n\right\}$ and $V(S(C_n))=\left\{v_1,v_2,\ldots v_n,v_1^\prime,v_2^\prime,\right.$ $\left. \ldots,
v_n^\prime\right\}$. By Theorem \ref{fertin1}, $\chi_s(C_n)=3$ 
and thus $\chi_s(S(C_n))\geq 3$, for $n\geq 6$. We claim that 3 colors are insufficient for star coloring 
of $S(C_n)$. First, we assign colors $\{c_1,c_2,c_3\}$ to vertices of $C_n$ to obtain proper star 3-coloring $c$ of $C_n$. If we 
want to extend this $c$ coloring into whole $S(C_n)$ without adding a new color, we may have at most three possibilities. 
Let $v_i, v_{i+1}, v_{i+2}, v_{i+3}$
be any four consecutive vertices of $C_n$, $1\leq i \leq n-3$. Since the cycle is star 3-colored, then exactly two out of four vertices
$v_i, v_{i+1}, v_{i+2}, v_{i+3}$ are assigned the same color:
\begin{enumerate}
\item $c(v_i)=c(v_{i+3})$
\\Then the coloring of $v_{i+1}^\prime$ and $v_{i+2}^\prime$ is determined. Vertices $v_{i+1}^\prime$ and $v_{i+2}^\prime$ must obtain the same 
colors as $v_{i+1}$ and $v_{i+2}$, respectively. But then 2-colored $P_4$ arises: $v_{i+2}^\prime, v_{i+1}, v_{i+2}, v_{i+1}^\prime$.
We need fourth color.
\item $c(v_i)=c(v_{i+2})$
\\It is clear that $c(v_{i+1})\neq c(v_{i+3})$ and the coloring of $v_{i+2}^\prime$ is determined while vertex $v_{i+1}^\prime$ can be assigned
two colors: $c(v_{i+1})$ or $c(v_{i+3})$. If we assign color $c(v_{i+1})$ to vertex $v_{i+1}^\prime$, then we get 2-chromatic 
$P_4$: $v_i,v_{i+1}^\prime, v_{i+2}, v_{i+1}$. In the second choice, also 2-chroma-tic $P_4$ arises: $v_i, v_{i+1}^\prime, v_{i+2}, v_{i+3}$. Also in
this case, we get additional color.
\item $c(v_{i+1})=c(v_{i+3})$
\\This case is analogous to case with $c(v_i)=c(v_{i+2})$.
\end{enumerate}
Summarizing, $\chi_s(S(C_n)\geq 4$. We will consider three cases depending on the value of $n \bmod 3$, $n\geq 3$.
\begin{description}
\item[Case 1] $n \equiv 0 \bmod 3$.
\\We define $4$-coloring $\sigma$ of $S(C_n)$ in the following way.
	
	$$\sigma(v_i)= \begin{cases} 
	c_1 \;\; \mbox{if} \;\; i\equiv 1 \bmod 3 \\ 
	c_2 \;\; \mbox{if} \;\; i\equiv 2 \bmod 3 \\ 
	c_3 \;\; \mbox{if} \;\; i\equiv 0 \bmod 3
	\end{cases}$$
and $\sigma(v^\prime_i)=c_4, 1\leq i\leq n$. 
Consider the color classes $c_i$ and $c_j$, $(1\leq i < j \leq 4)$. The components of induced subgraph of these color classes are $K_2$ and $K_{1,2}$. Hence there exists no bicolored path on four vertices and thus $\sigma$ is a proper star 4-coloring. Hence $\chi_s(S(C_n))=4$.\\

\item[Case 2] $n \equiv 1 \bmod 3$
\\For $n\geq 4$, we define $5$-coloring $\sigma$ of $S(C_n)$ in the following way.

	$$\sigma(v_i)= 
	\begin{cases} 
	c_1 \;\; \mbox{if} \;\; i\equiv 1 \bmod 3 \\ 
	c_2 \;\; \mbox{if} \;\; i\equiv 2 \bmod 3 \\ 
	c_3 \;\; \mbox{if} \;\; i\equiv 0 \bmod 3 
	\end{cases}$$
for $1 \leq i \leq n-1$, and $\sigma(v_n)=c_2$, 
	
	$$\sigma(v^\prime_i)= 
	\begin{cases} 
	c_4 \;\; \mbox{if} \;\; i\equiv 0 \bmod 3 \\ 
	c_5 \;\; \mbox{otherwise}
	\end{cases}$$
for $1\leq i\leq n$.

\par It is clear that every two out of three color classes corresponding to colors $c_1, c_2$, and $c_3$ avoid 2-chromatic $P_4$, 
similarly as
color classes corresponding to colors $c_4$ and $c_5$. We have to only check pairs of color classes $c_k;1 \leq k \leq 3$ and $c_4$ or $c_5$.
It is easy to check that every such two color classes induces forest whose components are isolated vertices, $P_2$ or $K_{1,2}$. Thus, $\sigma$ is proper star 5-coloring. Hence, $4 \leq \chi_s(S(C_n)) \leq 5$.

\item[Case 3] $n=5$
\\Star 5-coloring of $S(C_5)$ is given in Fig.~\ref{c8}b).

\item[Case 4] $n \equiv 2 \bmod 3$ and $n \geq 8$
\begin{description}
\item[Case 4.1] $n=8$
\\Star 4-coloring of $S(C_8)$ is given in Fig.~\ref{c8}c).

\begin{center}
\begin{figure}[htb]
\includegraphics[scale=0.3]{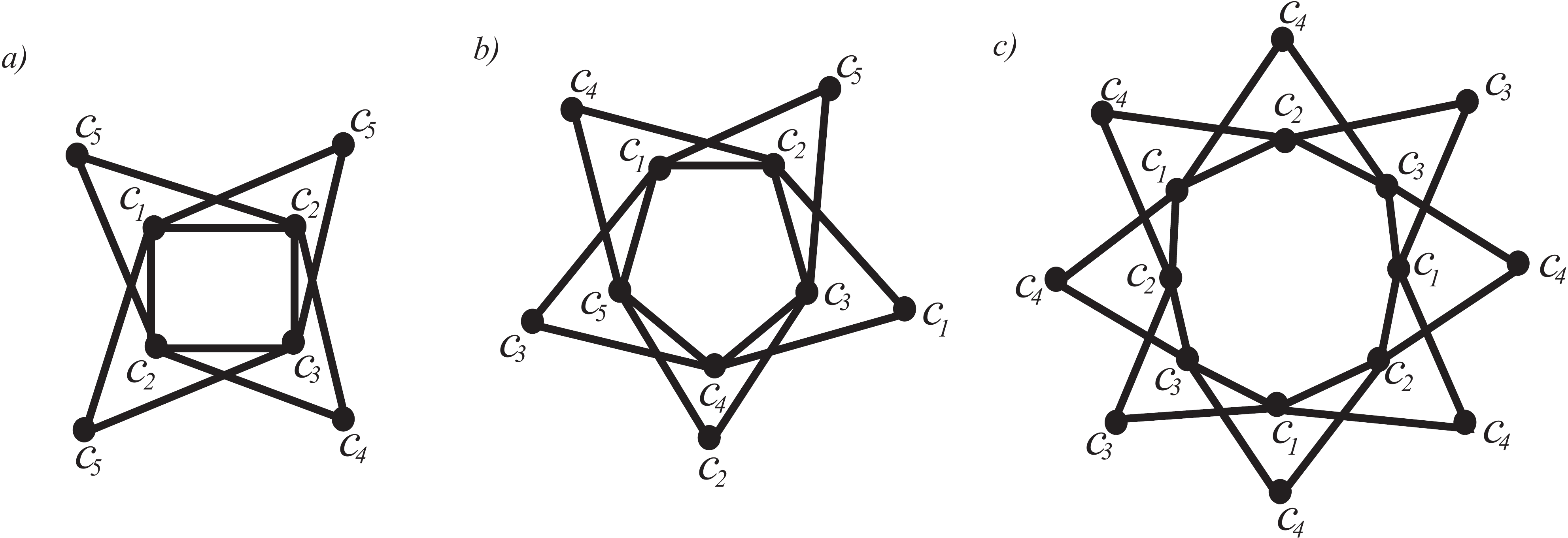}
\caption{a)$S(C_4)$; b) $S(C_5)$ and c)$S(C_8)$ with their star coloring.}\label{c8}
\end{figure}
\end{center}

\item[Case 4.2] $n \geq 11$
\\Let $n=8+3t$, $t \geq 1$. First, we color $v_1, \ldots v_8$ in the way given in Fig.~\ref{p8}. 

\begin{center}
\begin{figure}[htb]
\includegraphics[scale=0.3]{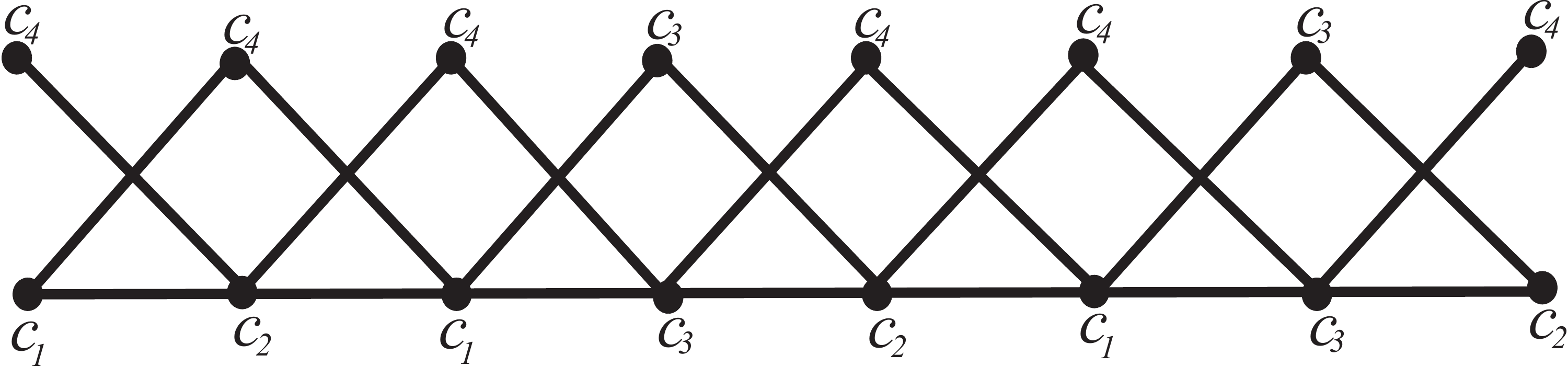}
\caption{A part of $S(C_n)$, $n \geq 11$, with its star 4-coloring.}\label{p8}
\end{figure}
\end{center}
Next, the remaining vertices of a cycle of $S(C_n)$ are colored in the following way.
	$$\sigma(v_i)= 
	\begin{cases} 
	c_1 \;\; \mbox{if} \;\; i\equiv 0 \bmod 3 \\ 
	c_2 \;\; \mbox{if} \;\; i\equiv 1 \bmod 3 \\ 
	c_3 \;\; \mbox{if} \;\; i\equiv 2 \bmod 3  
	\end{cases}$$
for $9\leq i\leq n$, and 
$$\sigma(v_i^\prime)= 
	\begin{cases} 
	c_3 \;\; \mbox{if} \;\; c(v_i)=3 \\ 
	c_4 \;\; \mbox{otherwise.} 
	\end{cases}$$

Similarly as it was in Case 1 we can check that $\sigma$ is proper star 4-coloring. Hence, $\chi_s(S(C_n))=4$.
\end{description}
\end{description}
\end{proof}

\section{Conclusion}
In the paper the problem of a star coloring for some splitting graphs has been considered. As an open question we put the problem 
of determining exact values of star chromatic number of splitting graphs of cycles $C_n$ where $n=5$ or $n \equiv 1 \bmod 3$. Moreover,
considering star coloring of degree splitting graphs, defined by Ponraj and Somasundaram \cite{ponraj}, seems to be worth paying attention.

\end{document}